\documentclass[12pt]{amsart}
\usepackage{amsmath,amsthm,amsfonts,amssymb,times,latexsym,mathabx,url}

\newtheorem{theorem}{Theorem}[section]

\newtheorem{lem}[theorem]{Lemma}

\numberwithin{equation}{section}

\newcommand{\e}{\varepsilon}
\renewcommand{\o}{\omega}
\renewcommand{\O}{\Omega}
\renewcommand{\a}{\alpha}

\renewcommand{\P}{\mathbb{P}}
\newcommand{\N}{\mathbb{N}}

\renewcommand{\leq}{\leqslant}
\renewcommand{\geq}{\geqslant}

\DeclareMathOperator{\R}{Re}
\DeclareMathOperator{\res}{res}

\makeatletter
\renewcommand{\pmod}[1]{\allowbreak\mkern7mu({\operator@font mod}\,\,#1)}
\makeatother

\newcommand{\be}{\begin{equation}}
\newcommand{\ee}{\end{equation}}

\renewcommand{\a}{\ensuremath{\alpha}}

\renewcommand{\le}{\leqslant}
\renewcommand{\leq}{\leqslant}

\renewcommand{\geq}{\geqslant}

\begin{document}

\title{Numbers of the form $kf(k)$}
\author{Mikhail R. Gabdullin, Vitalii V. Iudelevich, and Florian Luca}
\date{}

\address{Steklov Mathematical Institute,
	Gubkina str., 8, Moscow, Russia, 119991}
\email{gabdullin.mikhail@yandex.ru}

\address{Moscow State University, Leninskie Gory str., 1, Moscow, 119991}
\email{vitaliiyudelevich@mail.ru} 

\address{School of Maths, Wits University, Private Bag 3, Wits 2050, South Africa, Research Group in Algebraic Structures and Applications, King Abdulaziz University, Jeddah, Saudi Arabia and Centro de Ciencias Matem\'aticas UNAM, Morelia, Mexico}
\email{florian.luca@wits.ac.za}

\begin{abstract}
For a function $f\colon \N\to\N$, define $N^{\times}_{f}(x)=\#\{n\leq x: n=kf(k) \mbox{ for some $k$} \}$. Let $\tau(n)=\sum_{d|n}1$ be the divisor function, $\omega(n)=\sum_{p|n}1$ be the prime divisor function, and $\varphi(n)=\#\{1\leq k\leq n: (k,n)=1 \}$ be Euler's totient function. We prove that

\begin{gather*} 
\!\!\!\!\!\!\!\!\!\!\!\!\!\!\!\!\!\!\!\!\!\!\! 1) \quad N^{\times}_{\tau}(x) \asymp \frac{x}{(\log x)^{1/2}}; \\
2) \quad N^{\times}_{\omega}(x) = (1+o(1))\frac{x}{\log\log x}; \\
\!\!\!\!\!\!\!\!\! 3) \quad N^{\times}_{\varphi}(x) = (c_0+o(1))x^{1/2}, 
\end{gather*} 	
where $c_0=1.365...$\,.
\end{abstract}

\date{\today}
\maketitle


\section{Introduction}

Counting functions of various  sequences of positive integers have been extensively studied in number theory. A special case of great interest is that in which the sequence is the image of some arithmetic function. It is an easy consequence of the prime number theorem that a number $n\leq x$ can have at most $(1+o(1))\frac{\log x}{\log\log x}$ prime factors; therefore, if we denote by $\omega(n)=\sum_{p|n}1$ the number of prime divisors of $n$, we get
$$
\#\{\o(n): n\leq x \}=(1+o(1))\frac{\log x}{\log\log x}.
$$
The question becomes much more complex if we consider the divisor function $\tau(n)=\sum_{d|n}1$. In 1951, Erd\H{o}s and Mirsky \cite{EM} proved that
$$
\#\{\tau(n): n\leq x \} = \exp\left(\left(\frac{2\pi\sqrt2}{\sqrt3}+o(1)\right)\frac{(\log x)^{1/2}}{\log\log x} \right), 
$$
while it is not hard to see that $\max_{n\leq x}\tau(n)=\exp((\log2+o(1))\frac{\log x}{\log\log x})$ (see, for example, \cite{Ten}, Theorem I.5.4). Many papers were also devoted to the study of totients, that is, the numbers which are values of Euler's totient function $\varphi(n)=\#\{1\leq k\leq n: (k,n)=1 \}$. We just mention that Maier and Pomerance \cite{MP} (see this paper for the history of the question and references as well) in 1988 showed that
$$
\#\{\varphi(n): n\leq x \} = \frac{x}{\log x}\exp((C+o(1))(\log\log\log x)^2),
$$
and that the exact order of magnitude of the quantity $\#\{\varphi(n): n\leq x \}$ was found by Ford \cite{F} in 1998.

In the present paper we study counting functions of sequences of positive integers with the following special multiplicative structure. For a function $f\colon \N\to\N$, define 
$$
N^{\times}_{f}(x)=\#\{n\leq x: n=kf(k) \mbox{ for some $k$} \}.
$$ 
Note that all of the functions $\tau$, $\omega$, $\varphi$ have large typical values: all but $o(x)$ numbers $k\leq x$ have $\o(k)\asymp \log\log x$ (due to Hardy and Ramanujan \cite{HR}); hence, $\tau(n)\geq 2^{\o(n)}$ is usually also large, and, finally, $\varphi(k)\gg k/\log\log k$ for all $k$. Then it is easy to see that the corresponding counting functions $N^{\times}_{f}(x)$ for these $f$ are $o(x)$. A very natural question arises: what are their orders of magnitude? 

We give the answers for all of three mentioned choices of $f$. Firstly, we find the exact order of magnitude of $N^{\times}_{\tau}(x)$.

\begin{theorem}\label{th1.1}
We have
$$
N^{\times}_{\tau}(x) \asymp \frac{x}{(\log x)^{1/2}}.
$$ 	
\end{theorem}

Note that the map $k\mapsto k\tau(k)$ is not injective: we have $18\tau(18)=27\tau(27)$, and therefore, $18k\tau(18k)=27k\tau(27k)$ whenever $(k,6)=1$. Thus (as one can easily see from the proof of the lower bound) there is a positive proportion of the representable numbers which have at least two representations. This circumstance makes us think that it should be hard to find the asymptotics for $N^{\times}_{\tau}(x)$. However, our approach is well suited in the following cases where the map $k\mapsto kf(k)$ is an injection. Let $A\geq2$ be a fixed integer, and  define $f_1(n)=A^{\omega(n)}$ and $f_2(n)=A^{\Omega(n)}$ (here $\O(n)=\sum_{p^{\a}|n}1$ is the number of prime factors of $n$ counted with multiplicity); then, making some technical changes in the proof of Theorem \ref{th1.1} (and using Theorem 6.5 of \cite{Ten} in the case of $f_2(n)$), one can show that 
$$
N^{\times}_{f_i}(x)=(c_i(A)+o(1))\frac{x}{(\log x)^{1-1/A}},  \quad i=1,2,
$$  
where 
$$
c_1(A)=\frac{1}{\Gamma(1/A)}\prod_p\left(1+\frac{1}{A(p-1)}\right)\left(1-\frac1p\right)^{1/A},
$$
and
$$
c_2(A)=\frac{1}{\Gamma(1/A)}\prod_p\left(1-\frac{1}{Ap}\right)^{-1}\left(1-\frac1p\right)^{1/A}.
$$
Related results appear in \cite{BR} and \cite{LS}. In \cite{BR}, it is shown that if $\mu>0$ and $g(n)$ is a positive multiplicative function such that $g(n)\gg n^{-1/16}$ and $g(p)=1/\mu$ holds for all primes $p$, then there exists a positive constant $C$ (depending on $g$) such that
$$
\sum_{kg(k)\le x} 1=C x(\log x)^{\mu-1}+O_{\mu}(x\exp(-C(\log x)^{3/5} (\log\log x)^{-1/5})).
$$
In particular, the above result applies to $g(n)=\tau(n)$ with $\mu=1/2$ and to $g(n)=A^{\omega(n)}$ and $A^{\Omega(n)}$ with $\mu=1/A$. 
The paper \cite{LS} gives the order of magnitude of the counting function of the set of positive integers $n$ which are divisible by $A^{\omega(n)}$. 

\medskip 

Secondly, we study the case $f(n)=\o(n)$.

\begin{theorem}\label{th1.2}
We have
$$
N^{\times}_{\omega}(x) = \frac{x}{\log\log x} + O\left(\frac{x(\log\log\log x)^{1/2}(\log\log\log\log x)^2}{(\log\log x)^{3/2}}\right).
$$ 	
\end{theorem}

The map $k\mapsto k\o(k)$ is also not injective: for any prime $q\geq5$, we have $18q=9q\o(9q)=6q\o(6q)$. However, it turns out be very close to injective one, in the sense that the number of pairs $(k_1,k_2)$ with $k_1\neq k_2$ and $k_1\o(k_1)=k_2\o(k_2)$ is relatively small. This allows us to find the asymptotics for $N^{\times}_{\o}$.

The proof of this theorem can be easily adopted to include the case $f(n)=\O(n)$: we have the same asymptotics for $N^{\times}_{\Omega}$ as well.

Finally, we study the case $f=\varphi$. It turns out that the map $n\mapsto n\varphi(n)$ is an injection (see Section \ref{s4} for the details), and thus clearly $N^{\times}_{\varphi}(x)\geq \lfloor x^{1/2} \rfloor$. On the other hand, it is well-known that $\sum_{n\leq x}n/\varphi(n)\ll x$ (see, for example, [Mur], Exercise 4.4.12), and therefore for any $\e>0$ Markov's inequality implies that $\#\{n\leq x: \varphi(n)/n\leq \e \} \ll \e x$, which gives  
\begin{multline}\label{1.1}
N^{\times}_{\varphi}(x) \leq x^{1/2}+\sum_{j\geq0}\#\{n\in[2^jx^{1/2},2^{j+1}x^{1/2}]: \varphi(n)\leq 2^{-j}x^{1/2} \} \\
\leq x^{1/2}+\sum_{j\geq0}\#\{n\leq 2^{j+1}x^{1/2}: \frac{\varphi(n)}{n}\leq 4^{-j} \}
\ll x^{1/2}+\sum_{j\geq0}x^{1/2}2^{-j} \ll x^{1/2}.
\end{multline}
So, we see that $N^{\times}_{\varphi}(x)\asymp x^{1/2}$. The asymptotic behaviour of $N^{\times}_{\varphi}(x)$ is given in the following theorem.

\begin{theorem}\label{th1.3}
We have
$$
N^{\times}_{\varphi}(x) = c_0x^{1/2}+O(x^{1/2}\exp(-c\sqrt{\log x\log\log x})),
$$ 	
where $c_0=\prod_p\left(1+\frac{1}{p\,\left(p-1+\sqrt{p^2-p}\right)}\right)=1.365...$ and $c>0$ is an absolute constant.
\end{theorem}

It is worth mentioning that this last problem is very close to counting totient numbers up to $x$ with multiplicity. Let $r(n)=\#\{m\in \N: n=\varphi(m)\}$; in 1972, Bateman \cite{Bat} showed that
$$
\#\{n\in\N: \varphi(n)\leq x\} = \sum_{n\leq x}r(n)=\frac{\zeta(2)\zeta(3)}{\zeta(6)}x+O\left(x\exp(-c\sqrt{\log x\log\log x})\right)
$$
(here and in what follows $c$ stands for an absolute positive constant which may vary from line to line), and Balazard and Tenenbaum \cite{BT} in 1998 improved the error term to 
$$
O\left(x\exp(-c(\log x)^{3/5}(\log\log x)^{-1/5})\right),
$$
which is also the best known error term in the prime number theorem, due to Korobov \cite{Kor} and Vinogradov \cite{Vin}. It is very likely that the machinery of \cite{BT} may allow us to get the error term of the same shape in Theorem \ref{th1.3}, but we wanted to keep the paper short and self-contained, and thus decided to use a simpler argument which gives our result.

\medskip 

In Section \ref{s2}, we prove Theorem \ref{th1.1}; the main ingredients here are the asymptotics for the number of positive integers $k\leq x$ with a given value of $\o(k)$
and the asymptotics for the number of such square-free positive integers. Section \ref{s3} is devoted to the proof of Theorem \ref{th1.2}, which relies on the fact that the values $k\omega(k)$ are usually distinct for typical $k\leq x/\log\log x$. In Section \ref{s4}, we use the method of contour integration to prove Theorem \ref{th1.3}.

\bigskip 

\textbf{Notation.}  We use Vinogradov's $\ll$ notation: $F\ll G$ (as well as $F=O(G)$ and $G\gg F$) means that there exists an absolute constant $C>0$ such that $|F|\leq CG$; also we write $F\asymp G$ if $G\ll F\ll G$. We use $\lfloor u\rfloor$ to denote the largest integer not exceeding $u$, and we let $(a,b)$ be the greatest common divisor of integers $a$ and $b$.

\bigskip 

\textbf{Acknowledgements}. The authors thank Kevin Ford and Sergei Konyagin for helpful comments. Mikhail Gabdullin is supported in part by Young Russian Mathematics award. The work of Vitalii Iudelevich was supported by the Theoretical Physics and Mathematics Advancement Foundation  ``BASIS''.

\section{Proof of Theorem \ref{th1.1}}\label{s2}

We will need the following estimates.

\begin{lem}\label{lem2.1}
Let $Q(\a)=\a\log\a-\a+1$ and $\a_0>1$. Then
$$
\#\{n\leq x: \o(n)\leq \a\log\log x\} \ll x(\log x)^{-Q(\a)}
$$
for any $\a\in(0,1)$, and
$$
\#\{n\leq x: \o(n)\geq \a\log\log x\} \ll_{\a_0} x(\log x)^{-Q(\a)}
$$
for any $\a\in(1,\a_0]$.
\end{lem} 

\begin{proof}
	See \cite{HT}, Exercise 04.
\end{proof}

Let $\pi_l(x)=\#\{n\leq x: \o(n)=l \}$ and $\pi_l^*(x)=\#\{n\leq x: \o(n)=l \mbox{ and $n$ is square-free} \}$.

\begin{lem}\label{lem2.2}
Let $B>A>0$ be fixed. Then, for $x\geq3$ and $A\log\log x \leq l\leq B\log\log x$, 
$$
\pi^*_l(x) \asymp \pi_l(x) \asymp \frac{x}{\log x}\frac{(\log\log x)^{l-1}}{(l-1)!}.  
$$
\end{lem}

\begin{proof} Theorem II.6.4 of \cite{Ten} asserts that, for any $B>0$, $x\geq3$, and $1\leq l\leq B\log\log x$, we have (see the formula (6.18))
\begin{equation}\label{2.1}
\pi_l(x)=\frac{x}{\log x}\frac{(\log\log x)^{l-1}}{(l-1)!}\left\{\lambda\left(\frac{l-1}{\log\log x}\right)+O\left(\frac{l}{(\log\log x)^2}\right)\right\},
\end{equation}
where
$$
\lambda(z)=\frac{1}{\Gamma(z+1)}\prod_p\left(1+\frac{z}{p-1}\right)\left(1-\frac1p\right)^z
$$
is an entire function. Thus, if $A\log\log x\leq l\leq B\log\log x$, then $\lambda\left(\frac{l-1}{\log\log x}\right)\asymp 1$, and the claim for $\pi_l(x)$ follows.

The bounds for $\pi^*_l(x)$ follow from the analogue of (\ref{2.1}) for $\pi_l^*$, which can be proved similarly to Theorem II.6.4 in the book \cite{Ten}: starting with the function $$
\sum_{n=1}^{\infty}\frac{\mu^2(n)z^{\o(n)}}{n^s}
$$ 
(here $\mu$ is the M\"obius function) instead of 
$$
\sum_{n=1}^{\infty}\frac{z^{\o(n)}}{n^s},
$$ 
and applying Theorems 5.2 and 6.3 of \cite{Ten}, we get
$$
\pi^*_l(x)=\frac{x}{\log x}\frac{(\log\log x)^{l-1}}{(l-1)!}\left\{\lambda^*\left(\frac{l-1}{\log\log x}\right)+O\left(\frac{l}{(\log\log x)^2}\right)\right\},
$$
where
$$
\lambda^*(z)=\frac{1}{\Gamma(z+1)}\prod_p\left(1+\frac{z}{p}\right)\left(1-\frac1p\right)^z
$$
is another entire function. Again, if $A\log\log x\leq l\leq B\log\log x$, then $\lambda^*\left(\frac{l-1}{\log\log x}\right)\asymp 1$. This concludes the proof of the lemma.

\end{proof}

Now we are ready to prove Theorem \ref{th1.1}. We may assume that $x$ is large enough. We first prove the upper bound. For each representable $n\leq x$, we fix a $k$ with $k\tau(k)=n$; clearly, for any such $k$ we have $k2^{\o(k)}\leq k\tau(k)\leq x$. Thus
\begin{equation}\label{2.2}
N^{\times}_{\tau}(x) \leq \sum_{l\geq1} \pi_l(x/2^l)+1.
\end{equation}
Note that this sum is finite, since $\omega(k)\leq (1+o(1))\frac{\log x}{\log\log x}$ for any $k\leq x$. Let $y=\log\log x$. Lemma \ref{lem2.1} implies that
$$
\sum_{l\leq 0.1y}\pi_l(x/2^l) \leq \#\{k\leq x: \o(k)\leq 0.1y\} \leq x(\log x)^{-Q(0.1)} \ll \frac{x}{(\log x)^{0.6}},
$$
and we have
$$
\sum_{l\geq y}\pi_l(x/2^l) \leq \sum_{l\geq y}x/2^l \ll \frac{x}{2^y}\leq \frac{x}{(\log x)^{0.6}}.
$$
Using these two estimates and Lemma \ref{lem2.2} with $A=0.1$ and $B=1.1$, we get from (\ref{2.2}) 
\begin{multline*}
N^{\times}_{\tau}(x) \leq  \sum_{0.1y\leq l \leq y}\pi_l(x/2^l)+O\left(\frac{x}{(\log x)^{0.6}}\right) \ll \frac{x}{\log x}\sum_{0.1y \leq l\leq 
y} \frac{y^{l-1}}{2^l(l-1)!}+O\left(\frac{x}{(\log x)^{0.6}}\right)
\\=\frac{x}{2(\log x)^{1/2}}+O\left(\frac{x}{(\log x)^{1/2}}\sum_{|l-y/2|\geq 0.4y}\frac{(y/2)^{l}e^{-y/2}}{l!} + \frac{x}{(\log x)^{0.6}}\right),
\end{multline*}
since $y^{l-1}/(l-1)!\asymp y^l/l!$ for any $l\asymp y$. It is well-known (see, for example, (0.23) and (0.24) in \cite{HT}) that, for a Poisson random variable $\xi$ with parameter $y_0$, we have $\P(\xi\leq \a y_0)\leq e^{-Q(\a)y_0}$ for any $0\leq \a\leq 1$, and $\P(\xi \geq \a y_0)\leq e^{-Q(\a)y_0}$ for any $\a\geq 1$ (with $Q(\a)$ defined in Lemma \ref{lem2.1}). Using this with $y_0=y/2$, we get 
$$
\sum_{|l-y/2|\geq 0.4y}\frac{(y/2)^le^{-y/2}}{l!} \leq e^{-Q(0.2)y/2}+e^{-Q(1.8)y/2} \ll (\log x)^{-0.1},
$$
and the required upper bound for $N^{\times}_{\tau}(x)$ follows from the previous estimate.

To prove the lower bound, we note that if $k_1$ and $k_2$ are two distinct square-free numbers, then $k_1\tau(k_1)$ and $k_2\tau(k_2)$ are also distinct. Using Lemma \ref{lem2.2} and arguing as above, we have
$$
N^{\times}_{\tau}(x)\geq \sum_{l\geq1}\pi^*_l(x/2^l) \gg \frac{x}{\log x}\sum_{0.1y \leq l\leq y} \frac{y^l}{2^ll!} \gg \frac{x}{(\log x)^{1/2}}.
$$ 
This completes the proof of Theorem \ref{th1.1}.

\section{Proof of Theorem 1.2}\label{s3}

We need the following classical estimate.

\begin{lem}\label{lem3.1}
For any $0\leq \psi\leq \sqrt{\log\log x}$, we have
$$
\#\{k\leq x: |\o(k)-\log\log x|> \psi\sqrt{\log\log x}\} \ll xe^{-\frac13\psi^2}.
$$	
\end{lem}

\begin{proof} For the function $Q(\a)=\a\log\a-\a+1$, we have $|Q(1+\e)|\geq \e^2/3$ whenever $|\e|\leq1$. Now the claim follows from Lemma \ref{lem2.1} applied to $\a=1\pm \psi/\sqrt{\log\log x}$.
\end{proof}

To prove Theorem \ref{th1.2}, let us consider the following set of numbers. We assume that $x$ is large enough, and set\footnote{In this section, we use for brevity the notation $\log_2x=\log\log x$,  $\log_3x=\log\log\log x$, etc.}  $\psi=10(\log_3x)^{1/2}$. We define $K$ to be the set of positive integers $k$ such that
\begin{itemize}
	\item[(i)]  \qquad  $k\o(k)\leq x$;
	\item[(ii)] \qquad  $|\o(k)-\log\log x|\leq \psi(\log\log x)^{1/2}$.
\end{itemize}

\smallskip 

Now we briefly describe the idea of the proof. Firstly, due to Lemma \ref{lem3.1}, most numbers obey (ii), and thus, while counting the representable $n\leq x$, we can restrict our attention to those which are images of $k\in K$. Next, we show that the number of $n\leq x$ having more than one such representation is negligible. Therefore, $N^{\times}_{\omega}(x)\approx |K|$, and it remains to write down the asymptotics for $|K|$, which is $(1+o(1))\frac{x}{\log\log x}$.  
 
We turn to the details. The application of Lemma \ref{lem3.1} gives us
\begin{multline}\label{3.1}
\#\{k: k\o(k)\leq x \mbox { and } k\notin K\} \leq \#\left\{k\leq x: k \mbox{ violates (ii)} \right\} \\
\ll x\exp(-\psi^2/3) 
\ll \frac{x}{(\log\log x)^2}.
\end{multline}
Therefore,
\begin{equation}\label{3.2}
N^{\times}_{\omega}(x)=\#\{n\leq x: n=k\o(k) \mbox{ for some } k\in K\} + O\left(\frac{x}{(\log\log x)^2}\right).
\end{equation}

Let us call a number $n\leq x$ \textit{bad} if $n=k\o(k)=k'\o(k')$ for some distinct $k, k'\in K$; clearly, in this case $\o(k)$ and $\o(k')$ are distinct as well. Suppose we are given a bad $n$. Without less of generality, we may assume that
\begin{equation}\label{3.3}
\o(\o(k))\geq \o(\o(k')).
\end{equation}
Let $d=(k,k')$ and $t=(\o(k),\o(k'))$. The equality
$$
\frac{k\o(k)}{dt}=\frac{k'\o(k')}{dt}
$$
implies that 
$$
k=d\frac{\o(k')}{t}, \quad k'=d\frac{\o(k)}{t};
$$
therefore,
$$
\o(k)\leq \o(d)+\o(\o(k')) \leq \o(k')+\o(\o(k))
$$
and, similarly,
$$
\o(k') \leq \o(k)+\o(\o(k')).
$$
So, setting $u=\o(k')-\o(k)$, by (\ref{3.3}), (ii), and the bound $\o(m)\ll \log m$ (say) for all $m\in \N$, we have
\begin{equation}\label{3.4} 
0<|u|\leq \o(\o(k))\ll \log_3 x.	
\end{equation}
Note also that $t|u$ and, hence, $t\leq \o(\o(k))$. Let $w=\o(k)$. Since
$$
n=\frac{d\o(k')\o(k)}{t}=\frac{dw(w+u)}{t},
$$
we see that the number of bad $n$ does not exceed the number of the four-tuples $(w,t,d,u)$ under consideration. Let us fix $w$ and $t$. Then (ii) and (\ref{3.4}) implies that $w+u\asymp w$. Therefore, there are at most
$$
\frac{xt}{w(w+u)} \ll \frac{xt}{w^2} \ll \frac{xt}{(\log_2x)^2}
$$  
possible values of $d$. Further, by (\ref{3.4}) there are 
$$
\ll \frac{\o(w)}{t}
$$
possible values of $u$. Finally, there are at most $\max_{1\leq u\leq \o(w)} \tau(u) \leq \o(w)$ options for $t$ for any fixed $w$. Combining all of this, we see that the number of bad $n$ does not exceed
$$
\frac{x}{(\log_2x)^2}\sum_{a\leq w\leq b}\o^2(w), 
$$
where $a=\log_2x-\psi(\log_2x)^{1/2}$ and $b=\log_2x+\psi(\log_2x)^{1/2}$. Since 
$$
\o(w)=\sum_{p|w: p\leq b^{1/10}} 1 + O(1)
$$ 
for any $w\in[a,b]$, and $b-a>b^{1/5}$, we find that
\begin{multline*}
\sum_{a\leq w\leq b}\o^2(w)	= \sum_{p,q \leq b^{1/10}: \, p\neq q}\frac{b-a}{pq}+O\left(\sum_{p\leq b^{1/10}}\frac{b-a}{p}+(b-a)\right) \\
\ll (b-a)(\log_4x)^2\ll \psi(\log_2x)^{1/2}(\log_4x)^2.
\end{multline*}
Thus, the number of bad $n$ is  
$$
\ll \frac{x(\log_3x)^{1/2}(\log_4x)^2}{(\log_2x)^{3/2}}.
$$
Now it follows from (\ref{3.2}) that
\begin{equation}\label{3.5}
N_{\o}^{\times}(x)=|K|+O\left(\frac{x(\log_3x)^{1/2}(\log_4x)^2}{(\log_2x)^{3/2}}\right).	
\end{equation}

Finally, we work with $|K|$. By (i) and (ii), any $k\in K$ does not exceed
$$
\frac{x}{\log_2 x-\psi(\log_2 x)^{1/2}}=\frac{x}{\log_2 x}+O\left(\frac{x(\log_3x)^{1/2}}{(\log_2 x)^{3/2}}\right),
$$
and thus
$$
|K|\leq \frac{x}{\log_2 x}+O\left(\frac{x(\log_3x)^{1/2}}{(\log_2 x)^{3/2}}\right).
$$
On the other hand, any $k$ not exceeding 
$$
\frac{x}{\log_2 x+\psi(\log_2 x)^{1/2}}=\frac{x}{\log_2 x}-O\left(\frac{x(\log_3x)^{1/2}}{(\log_2 x)^{3/2}}\right)
$$
and obeying (ii), belongs to $K$. Using the bound (\ref{3.1}), we find that
$$
|K|\geq  \frac{x}{\log_2 x}-O\left(\frac{x(\log_3x)^{1/2}}{(\log_2 x)^{3/2}}\right).
$$
So
$$
|K|=\frac{x}{\log_2 x}+O\left(\frac{x(\log_3x)^{1/2}}{(\log_2 x)^{3/2}}\right),
$$
and now Theorem \ref{th1.2} follows from (\ref{3.5}).

\section{Proof of Theorem 1.3}\label{s4}

We first note that the map $k\mapsto k\varphi(k)$ is an injection. Indeed, let $n=k\varphi(k)=l\varphi(l)$ and $p=P^+(n)$; then it is easy to see that $P^+(k)=P^+(l)=p$ and $p$ occurs in $k$ and $l$ in the same power, say, $\a$. Thus we can divide the equality $k\varphi(k)=l\varphi(l)$ by $p^{2\a-1}(p-1)$ and get $k'\varphi(k')=l'\varphi(l')$, where $k'=k/p^{\a}$ and $l'=l/p^{\a}$ are coprime to $p$. Arguing in the same manner, we obtain $k=l$ after a finite number of steps. 

Now we consider the function 
\begin{equation*}
F(s)=\prod_{p}\left(1+\frac{1}{(p\varphi(p))^s}+\frac{1}{(p^2\varphi(p^2))^s}+\frac{1}{(p^3\varphi(p^3))^s}+\ldots \right);
\end{equation*}
since $\varphi(p^{\a})\asymp p^\a$, we see that $F(s)$ absolutely converges in $\R s>1/2$. Denote 
$$
A=\{n\in\N: n=k\varphi(k) \mbox { for some } k\}.
$$ 
Since any $n$ has at most one such representation, we have 
$$
F(s)=\sum_{n=1}^{\infty}\frac{\mathbb{I}(n\in A)}{n^s}.
$$
Further,
\begin{multline}\label{4.1}
F(s)=\prod_p\left(1+\frac{1}{(p-1)^s}\left(\frac{1}{p^s}+\frac{1}{p^{3s}}+
\frac{1}{p^{5s}}+\ldots\right)\right)\\
=\prod_p\left(1+\frac{1}{(p-1)^sp^s(1-p^{-2s})}\right)=\zeta(2s)G(s),
\end{multline}
where $\zeta(s)$ is the Riemann zeta-function and
\begin{equation}\label{4.2}
G(s)=\prod_p\left(1+\frac{p^s-(p-1)^s}{p^{2s}(p-1)^s}\right).
\end{equation}
Since $p^s-(p-1)^s=s\int_{p-1}^pu^{s-1}du$, for any $s$ with $\R s=\sigma>0$ we have $$
\left|\frac{p^s-(p-1)^s}{p^{2s}(p-1)^s}\right|\ll |s|p^{-(1+2\sigma)}.
$$ 
Thus, $G(s)$ is analytic in $\R s>0$.

We use Perron's formula to find the asymptotics of $\int_1^xN^{\times}_{\varphi}(u)du=\int_1^x\#\{n\leq u: n\in A\}du$, which will imply the asymptotics for $N^{\times}_{\varphi}(x)$.

\begin{lem}[Perron's formula; see \cite{KV}, Appendix, \S5, Theorem 2]\label{lem4.1}
Let 
$$
F(s)=\sum_{n=1}^{\infty}\frac{a(n)}{n^s}
$$ 
be a Dirichlet series which absolutely converges in $\R s >a_0\geq0$ and $A(u)=\sum_{n\leq u}a(n)$. For $b>a_0$, define 	
$$
B(b)=\int_1^{\infty}\frac{|A(u)|}{u^{b+1}}du.
$$
Then, for all $x\geq2$ and $T\geq2$, 
$$
\int_1^xA(u)du=\frac{1}{2\pi i}\int_{b-iT}^{b+iT}\frac{F(s)x^{s+1}}{s(s+1)}ds+R(x),
$$
where
$$
R(x) \ll B(b)\frac{x^{b+1}}{T}+2^b\left(\frac{x\log x}{T}+\log T\right)\max_{x/2\leq u\leq 3x/2}|A(u)|.
$$
and the implied constant is absolute.
\end{lem}

We apply this for $a(n)=\mathbb{I}(n\in A)$ (so $A(x)=N^{\times}_{\varphi}(x)$), $a_0=1/2$, large enough $x$, and $b=1/2+1/\log x$; we also assume that $10\leq T\leq x$. Let us estimate the error term $R(x)$. By (\ref{1.1}), we have $A(u)\ll u^{1/2}$ and hence,
$$
B(b)\ll \int_1^{\infty}\frac{du}{u^{b+1/2}} \ll \log x
$$
and $R(x)\ll \frac{x^{3/2}\log x}{T}$. Thus, by Lemma \ref{lem4.1} and (\ref{4.1}),
\begin{equation}\label{4.3}
\int_1^xA(u)du=\frac{1}{2\pi i}\int_{b-iT}^{b+iT}\frac{\zeta(2s)G(s)x^{s+1}}{s(s+1)}ds+O\left(\frac{x^{3/2}\log x}{T}\right).
\end{equation}

We compute this integral using Cauchy's theorem. Setting
$$
a=\frac12-\frac{\log\log T}{5\log T},
$$ 
consider the contour $\Gamma=\Gamma_1\cup\Gamma_2\cup\Gamma_3\cup\Gamma_4$, where $\Gamma_1=[b+iT,a+iT]$,  $\Gamma_2=[a+iT,a-iT]$, $\Gamma_3=[a-iT,b-iT]$,  $\Gamma_4=[b-iT,b+iT]$. We also write 
$$
I_i=\int_{\Gamma_i}\frac{\zeta(2s)G(s)x^{s+1}}{s(s+1)}ds
$$ 
for $1\leq i\leq 4$. It is easy to see that the integrand has one simple pole at the point $s=1/2$ and, since 
$$
\zeta(2s)=\frac{1}{2(s-1/2)}(1+o(1))
$$ 
as $s\to1/2$, we have 
$$
\int_{\Gamma}\frac{\zeta(2s)G(s)x^{s+1}}{s(s+1)}ds = \res_{s=1/2}\frac{\zeta(2s)G(s)x^{s+1}}{s(s+1)} = \frac{G(1/2)x^{3/2}}{3/2}=\frac23G(1/2)x^{3/2}.
$$
Hence, (\ref{4.3}) implies 
\begin{equation}\label{4.4}
\int_1^xA(u)du=\frac23G(1/2)x^{3/2}+O\left(|I_1|+|I_2|+|I_3|+\frac{x^{3/2}\log x}{T}\right).
\end{equation}
Now we estimate the integrals in the error term. 
Firstly, we estimate the function $G(s)$ for $s\in \Gamma_1\cup\Gamma_2\cup\Gamma_3$ with $|t|\geq2$. Since 
$$
p^s-(p-1)^s=s\int_{p-1}^pu^{s-1}ds,
$$ 
for any $s$ with $\sigma=\R s\in(0,1)$ we have
$$
|p^s-(p-1)^s|\leq \min\{2p^{\sigma}, |s|(p-1)^{\sigma-1} \}.
$$
Clearly, the first bound is better iff $p\ll |s|\asymp |t|$. Thus, from the definition (\ref{4.2}) of $G(s)$ we get
$$
|G(s)|\leq \prod_{p\leq |t|}\left(1+O(p^{-2\sigma})\right) \prod_{p>|t|}\left(1+O\left(|t|p^{-1-2\sigma}\right)\right),
$$
and, since $\sigma\geq a=1/2-\log\log T/(5\log T)$,
$$
\log |G(s)| \ll \sum_{p\leq |t|}p^{-2a} + |t|\sum_{p>|t|}p^{-1-2a}.
$$
Since $|t|\leq T$, for any $p\leq |t|$ we get $p^{-2a}\leq p^{-1}(\log T)^{0.4}$. Thus,
$$
\log |G(s)| \ll (\log T)^{0.4}\sum_{p\leq |t|}p^{-1} + |t|^{1-2a} \ll (\log T)^{1/2}
$$
and, hence,
$$
|G(s)| \leq \exp(O((\log T)^{1/2}))
$$
for any $s\in \Gamma_1\cup\Gamma_2\cup\Gamma_3$.  We also need the following well-known bound for the Riemann zeta-function [see \cite{KV}, Theorem 2 of Chapter IV]: for small enough positive $\gamma_1$ and $\sigma\geq 1-\gamma_1/(\log|t|)^{2/3}$, $|t|\geq2$, we have 
$$
\zeta(\sigma+it)=O\left(\log^{2/3}|t|\right).
$$

Using these bounds, we get 
$$
\max\{|I_1|, |I_3|\} \ll \int_a^b\frac{(\log T)^{2/3}\exp(O((\log T)^{1/2}))x^{1+\sigma}}{T^2}d\sigma \ll \frac{x^{3/2}}{T},
$$
and
$$
|I_2| \ll x^{1+a}e^{O((\log T)^{1/2})}\left(\int_2^{T}\frac{(\log t)^{2/3}}{t^2}dt +O(1)\right) \ll x^{3/2-\log\log T/(5\log T)}e^{O((\log T)^{1/2})}.
$$
So, we have from (\ref{4.4})
$$
\int_1^xA(u)du = \frac23G(1/2)x^{3/2}+O\left(\frac{x^{3/2}\log x}{T} + x^{3/2-\log\log T/(5\log T)} e^{O((\log T)^{1/2})}\right).
$$
Choosing $T=\exp(c(\log x \log\log x)^{1/2})$ for some $c>0$, we have
\begin{equation}\label{4.5} 
\int_1^xA(u)du = \frac23G(1/2)x^{3/2}+O\left(x^{3/2}\exp(-c_1(\log x \log\log x)^{1/2}) \right)
\end{equation}
for some $c_1>0$, and simple calculations show that 
$$
G(1/2)=\prod_p\left(1+\frac{1}{p(p-1+\sqrt{p^2-p})}\right)=1.365...
$$

Now we complete the proof by the standard ``differentiation'' of the above asymptotic formula. Let $1\leq h\leq x/2$. Then		
$$
\int_{x-h}^xA(u)du \leq hA(x) \leq \int_x^{x+h}A(u)du.
$$
On the other hand, (\ref{4.5}) implies
$$
\int_x^{x+h}A(u)du=\frac23G(1/2)\left((x+h)^{3/2}-x^{3/2}\right)+O\left(x^{3/2}\exp(-c_1(\log x \log\log x)^{1/2})\right),
$$
and it is easy to see that
$$
(x+h)^{3/2}-x^{3/2}=x^{3/2}\left(\frac{3h}{2x}+O\left(\frac{h^2}{x^2}\right)\right)=\frac32x^{1/2}h+O\left(\frac{h^2}{x}\right).
$$
The last three estimates yield
$$
A(x)\leq G(1/2)x^{1/2}+O\left(\frac{h}{x^{1/2}}+\frac{x^{3/2}\exp(-c_1(\log x \log\log x)^{1/2})}{h}\right).
$$
Now we choose $h=x\exp(-0.5c_1(\log x \log\log x)^{1/2})$ and obtain
$$
A(x)\leq G(1/2)x^{1/2}+O\left(x^{1/2}\exp(-0.5c_1(\log x \log\log x)^{1/2})\right).
$$
Similarly, one can show that
$$
A(x)\geq G(1/2)x^{1/2}+O\left(x^{1/2}\exp(-0.5c_1(\log x \log\log x)^{1/2})\right).
$$
This completes the proof.

\end{document}